\documentclass[12pt,english]{smfart}

\usepackage{smfthm}
\usepackage{amssymb}

\newcommand{\Qp}{\mathbf{Q}_p}
\newcommand{\Qpp}{\mathbf{Q}_{p^2}}
\newcommand{\Zp}{\mathbf{Z}_p}
\newcommand{\Zpp}{\mathbf{Z}_{p^2}}

\newcommand{\ZZ}{\mathbf{Z}}
\newcommand{\OO}{\mathcal{O}}
\newcommand{\OL}{\mathrm{O}}
\newcommand{\Qpbar}{\overline{\mathbf{Q}}_p}
\newcommand{\eps}{\varepsilon}
\renewcommand{\phi}{\varphi}

\renewcommand{\geq}{\geqslant}
\renewcommand{\leq}{\leqslant} 
\newcommand{\Gal}{\operatorname{Gal}}
\newcommand{\galp}{\Gal(\Qpbar/\Qp)}
\newcommand{\galpp}{\Gal(\Qpbar/\Qpp)}
\newcommand{\dfont}{\mathrm{D}}
\newcommand{\dcris}{\mathrm{D}_{\operatorname{cris}}}
\newcommand{\indpp}{\operatorname{Ind}_{\Qpp}^{\Qp}}
\newcommand{\Mat}{\operatorname{Mat}}

\newcommand{\Aut}{\operatorname{Aut}}
\newcommand{\Hom}{\operatorname{Hom}}
\newcommand{\GL}{\operatorname{GL}}
\newcommand{\LT}{\operatorname{LT}}

\newcommand{\pmat}[1]{\begin{pmatrix} #1 \end{pmatrix}}

\newcommand{\dcroc}[1]{[\![ #1 ]\!]}
\newcommand{\calE}{\mathcal{E}}
\newcommand{\calR}{\mathcal{R}}
\newcommand{\diam}[1]{\langle #1 \rangle}
 
\newcommand{\bcris}{\mathbf{B}_{\operatorname{cris}}} 
\newcommand{\bdr}{\mathbf{B}_{\operatorname{dR}}}  
\newcommand{\Fil}{\operatorname{Fil}}
\newcommand{\res}{\operatorname{res}}

\author{Laurent Berger}
\address{Universit\'e de Lyon \\
UMPA ENS Lyon \\
46 all\'ee d'Italie \\
69007 Lyon \\
France}
\email{laurent.berger@umpa.ens-lyon.fr}
\urladdr{www.umpa.ens-lyon.fr/\~{}lberger/}

\date{March 2009}

\title{A $p$-adic family of dihedral $(\varphi,\Gamma)$-modules}
\alttitle{Une famille $p$-adique de $(\varphi,\Gamma)$-modules di\'edraux}

\subjclass{11F33, 11F80, 11S20, 11S31, 12H25, 14D15, 14F30}

\keywords{$(\varphi,\Gamma)$-modules, Galois representations, Lubin-Tate characters, $p$-adic Hodge theory, Wach modules, trianguline representations, $p$-adic families, $p$-adic modular forms}

\begin{document}

\begin{abstract}
The goal of this article is to construct explicitly a $p$-adic family of representations (which are dihedral representations), to construct their associated $(\varphi,\Gamma)$-modules by writing down explicit matrices for $\varphi$ and for the action of $\Gamma$, and finally to determine which of these are trianguline.
\end{abstract}

\begin{altabstract}
L'objet de cet article est de construire explicitement une famille $p$-adique de repr\'esentations (qui sont des repr\'esentations di\'edrales), de construire les $(\varphi,\Gamma)$-modules qui leurs sont associ\'es en donnant des matrices explicites pour $\varphi$ et pour l'action de $\Gamma$, et finalement de d\'eterminer lesquelles sont triangulines.
\end{altabstract}

\maketitle

\tableofcontents

\setlength{\baselineskip}{18pt}

\section*{Introduction}
A fundamental tool in the theory of $p$-adic representations is Fontaine's construction in \cite{F90} of the $(\phi,\Gamma)$-modules attached to $p$-adic representations. These are modules over a ring of power series, and are very explicit objects which contain all of the information about the representations they are attached to. It is however not always easy to extract that information. The work of Cherbonnier-Colmez \cite{CC} and Kedlaya \cite{K5} has done much to clarify the situation, and in particular has allowed us to understand the structure of the $(\phi,\Gamma)$-modules attached to semistable representations as in \cite{LB2} and \cite{LB8}. Inspired by these constructions and the $p$-adic Langlands correspondence, Colmez has defined in \cite{C08} the notion of trianguline representation, and reinterpreted the first main result of \cite{K3} as saying that the $p$-adic representations coming from finite slope overconvergent $p$-adic modular forms are trianguline.

The goal of this modest article is to construct explicitly a $p$-adic family of representations, which are dihedral representations, to construct their $(\phi,\Gamma)$-modules by writing down explicit matrices for $\phi$ and for the action of $\Gamma$, and finally to determine which of these are trianguline. This can be seen as a very first step towards constructing a universal family of $(\phi,\Gamma)$-modules corresponding to the universal deformation space (in the sense of \cite{M89}) of a mod $p$ representation of $\galp$. Our results extend without trouble to potentially abelian representations, but the general case will require new ideas.

We now give a more precise description of our results. Let $\chi_2 : \galpp \to \Zpp^\times$ be the character associated to the Lubin-Tate module over $\Zpp$ for the uniformizer $p$ of $\Qpp$. Every element $x \in \Zpp^\times$ can be written in a unique way as $x= \omega(x)\diam{x}$ where $\omega(x)^{p^2-1} =1$ and $\diam{x} \in 1+p\Zpp$. If $g \in \galpp$, then we define $\omega_2(g) = \omega(\chi_2(g))$ and $\diam{g}_2 = \diam{\chi_2(g)}$. Since $\diam{g}_2 \in 1+p\Zpp$, the expression $\diam{g}_2^s$ makes sense if $s\in \Zp$ and the representations $\indpp(\omega_2(\cdot)^r \diam{\cdot}_2^s)$ $p$-adically interpolate the $\indpp(\chi_2^h)$ with $h\in \ZZ$. Our first result is an explicit construction of the $(\phi,\Gamma)$-modules associated to these representations, and we briefly recall what these objects are. 

The Robba ring is the ring $\calR$ consisting of power series $f(X)=\sum_{k \in \ZZ} a_k X^k$ with $a_k \in \Qp$ and such that $f(X)$ converges on an annulus $r_f < |X|_p < 1$ where $r_f$ depends on $f$. This ring is endowed with a frobenius $\phi$ given by $\phi(f)(X)=f((1+X)^p-1)$, and with an action of $\Gamma=\Gal(\Qp(\mu_{p^\infty})/\Qp)$ given by $\gamma(f)(X)=f((1+X)^{\chi(\gamma)}-1)$ where $\chi : \Gamma \to \Zp^\times$ is the cyclotomic character. A $(\phi,\Gamma)$-module over $\calR$ is a free module of finite rank over $\calR$ endowed with a semilinear frobenius $\phi$ such that $\Mat(\phi) \in \GL_d(\calR)$ and a semilinear continuous action of $\Gamma$ commuting with $\phi$. By combining the aforementioned constructions of Fontaine, Cherbonnier-Colmez and Kedlaya, we get a functor $V \mapsto \dfont(V)$ which to every $p$-adic representation associates a $(\phi,\Gamma)$-module over $\calR$. This functor gives an equivalence of categories between $p$-adic representations and \'etale $(\phi,\Gamma)$-modules over $\calR$. Finally, let $Q=\phi(X)/X=((1+X)^p-1)/X$ and let $Q_2=\phi(Q)$ so that in some suitable sense, $Q_2/Q^p=1\bmod{p}$ and the expression $(Q_2/Q^p)^u$ makes sense if $u \in \Zp$.

\begin{enonce*}{Theorem A}
If $r \in \ZZ$ and $u \in \Zp$ and $s=r+(p+1)u$, then the $(\phi,\Gamma)$-module associated to $\indpp(\omega_2(\cdot)^r \diam{\cdot}_2^s)$ has a basis in which
\[ \Mat(\phi) = \pmat{0 & 1 \\ Q_2^r (Q_2/Q^p)^u & 0} \]
and
\[ \Mat(\gamma) = \pmat{\chi(\gamma)^r \diam{\chi(\gamma)}^u (1+\OL(X)) & 0 \\ 0 & \chi(\gamma)^r \diam{\chi(\gamma)}^u (1+\OL(X))}, \]
where the $1+\OL(X)$ are two power series belonging to $1+X\Zp\dcroc{X}$.
\end{enonce*}

The proof of this result (theorem \ref{pgmisind}) is by $p$-adic interpolation. If $h \in \ZZ$, then the representation $\indpp(\chi_2^h)$ is crystalline and we can compute its $(\phi,\Gamma)$-module by using the theory of Wach modules of \cite{LB4}. One then only needs to change the basis so that the matrices of $\phi$ and $\gamma \in \Gamma$ become continuous functions of $h$. 

One can then work concretely with the representation $\indpp(\omega_2(\cdot)^r \diam{\cdot}_2^s)$ and our next result (theorem \ref{indtrig}) tells us exactly when it is trianguline. A $p$-adic representation is said to be trianguline if its associated $(\phi,\Gamma)$-module over $\calR$ is an iterated extension of $(\phi,\Gamma)$-modules of rank $1$, after possibly extending scalars.

\begin{enonce*}{Theorem B}
The representation $\indpp(\omega_2(\cdot)^r \diam{\cdot}_2^s)$ is trianguline if and only if $s \in \ZZ$ and $r=s \bmod{p+1}$.
\end{enonce*}

In particular, by combining theorem 6.3 of \cite{K3} and proposition 4.3 of \cite{C08}, we see that if $s \notin \ZZ$, then the representation $\indpp(\omega_2(\cdot)^r \diam{\cdot}_2^s)$ does not arise from a finite slope overconvergent $p$-adic modular form. 

It is not hard to analytify our constructions and hence to get a two-dimensional representation over $\Zp\{T\}$. An analogue of theorem A then gives a corresponding family of $(\phi,\Gamma)$-modules over $\Zp\{T\}$ and theorem B tells us about the trianguline locus for that family. Note that one can twist $\indpp(\omega_2(\cdot)^r \diam{\cdot}_2^s)$ by a character of $\galp$ and this way one obtains a three-dimensional family of representations, sitting inside the usually five-dimensional (see \S9 of \cite{FM}) universal deformation space of a given mod $p$ representation. These families are in some sense orthogonal to the ones constructed in \cite{BLZ}. Can one combine them to get an explicit family over some four-dimensional space? 

\section{A family of dihedral representations}
We start by constructing the representations $\indpp(\omega_2(\cdot)^r \diam{\cdot}_2^s)$ in a way which shows that they are actually defined over $\Qp$. Let $\chi_2 : \galpp \to \Zpp^\times$ be the character associated to the Lubin-Tate module over $\Zpp$ for the uniformizer $p$ of $\Qpp$ and denote by $\Qpp^{\LT}$ the fixed field of $\ker(\chi_2)$. If $\sigma : \Zpp \to \Zpp$ denotes the absolute frobenius, then the group $\Gal(\Qpp^{\LT} / \Qp)$ is naturally isomorphic to $\Zpp^\times \rtimes \ZZ/2\ZZ$, where the map $\ZZ/2\ZZ \to \Aut(\Zpp^\times)$ is given by $\eps \mapsto \sigma^\eps$.

Every element $x \in \Zpp^\times$ can be written in a unique way as $x= \omega(x)\diam{x}$ where $\omega(x)^{p^2-1} =1$ and $\diam{x} \in 1+p\Zpp$. If $g \in \galpp$, then we define $\omega_2(g) = \omega(\chi_2(g))$ and $\diam{g}_2 = \diam{\chi_2(g)}$. If $r \in \ZZ/(p^2-1)\ZZ$ and $s \in \Zp$, then we have a character $\eta_{r,s} : \Zpp^\times \to \Zpp^\times$ given by $x \mapsto  \omega(x)^r \diam{x}^s$ where 
\begin{align*}
\diam{x}^s & = (1+(\diam{x}-1))^s \\
& = \sum_{k \geq 0} \binom{s}{k}(\diam{x}-1)^k \in 1+ p \Zpp.
\end{align*}
If $d \in \Zpp$ is some element such that $\Zpp=\Zp[\sqrt{d}]$, then we have a homomorphism $\Zpp^\times \rtimes \ZZ/2\ZZ \to \GL_2(\Zp)$ given by 
\[ (x+y\sqrt{d},0) \mapsto \pmat{x & dy \\ y & x} \quad\text{and}\quad 
(x+y\sqrt{d},1) \mapsto \pmat{x & -dy \\ y & -x}.  \] 
By composing $\chi_2$ and $\eta_{r,s}$ we get a representation 
\[ \rho_{r,s} : \galp \to \GL_2(\Zp), \]
whose underlying $\Zp$-module we denote by $T_r(s)$. We also let $V_r(s) = \Qp \otimes_{\Zp} T_r(s)$.

\begin{lemm}\label{repcong}
If $s_1 = s_2 \bmod{p^k}$, then $T_r(s_1) = T_r(s_2) \bmod{p^{k+1}}$.
\end{lemm}

\begin{proof}
If $s_1 = s_2 \bmod{p^k}$, then $(1+(\diam{x}-1))^{s_1} = (1+(\diam{x}-1))^{s_2}\bmod{p^{k+1}}$ for all $x \in \Zpp^\times$ and therefore the same is true of $\rho_{r,s_1}$ and $\rho_{r,s_2}$.
\end{proof}

\begin{prop}\label{specisind}
We have $\Qpp \otimes_{\Qp} V_r(s) = \indpp(\omega_2(\cdot)^r \diam{\cdot}_2^s)$.
\end{prop}

\begin{proof}
In a suitable basis of $\Qpp \otimes_{\Qp} V_r(s)$, this representation is isomorphic to $\omega_2(\cdot)^r \diam{\cdot}_2^s \oplus \sigma(\omega_2(\cdot)^r \diam{\cdot}_2^s)$. If $s=0$ and $r=0\bmod{p+1}$, then the proposition is obvious; otherwise, the two characters $\omega_2(\cdot)^r \diam{\cdot}_2^s$ and $\sigma(\omega_2(\cdot)^r \diam{\cdot}_2^s)$ are distinct so that $\indpp(\omega_2(\cdot)^r \diam{\cdot}_2^s)$ is irreducible by Mackey's criterion and the proposition follows from Frobenius reciprocity.
\end{proof}
 
\section{Crystalline periods for Lubin-Tate groups}
The character $\chi_2$ extends to a map $\chi_2 : \galp \to \Zpp^\times$ which is no longer a character but satisfies the formula $\chi_2(gh)=g(\chi_2(h))\chi_2(g)$. Let $\bcris$ and $\bdr$ be the rings of periods constructed in \cite{F3} and let $t_2 \in \bcris^+$ be the element $t_E$ constructed in \S 9.3 of \cite{C02} for $E=\Qpp$ and $\pi_E=p$. 

\begin{prop}\label{prt}
The element $t_2$ has the following properties :
\begin{enumerate}
\item if $g \in \galp$, then $g(t_2)=\chi_2(g)t_2$;
\item $\phi^2(t_2)=pt_2$;
\item $t_2 \in \Fil^1 \setminus \Fil^2 \bdr$ and $\phi(t_2) \in \Fil^0 \setminus \Fil^1 \bdr$.
\end{enumerate}
\end{prop}

\begin{proof}
Properties (2) and (3) are proved in \S 2.4 of \cite{CEV}. As for property (1), we use the notations of \S  9 of \cite{C02}. The element $t_E$ is defined as $L_E(\omega_E)$ where $\omega_E$ is constructed so that $g(\omega_E) = [\chi_2(g)] (\omega_E)$ and $L_E$ is the logarithm of the Lubin-Tate group, which implies (1).
\end{proof}

In particular, if $h \in \ZZ$, then the space $W_h = \Qpp \cdot t_2^h$ is a $\galp$-stable subspace of $\bcris$ and hence a two-dimensional representation of $\galp$. 

\begin{lemm}\label{whisind}
We have $\Qpp \otimes_{\Qp} W_h  = \indpp(\chi_2^h)$.
\end{lemm}

\begin{proof}
The lemma is immediate if $h=0$, so let us assume that $h \neq 0$. The restriction of $W_h$ to $\galpp$ contains the characters $\chi_2^h$ and $\sigma(\chi_2)^h$ and since $\chi_2^h \neq \sigma(\chi_2)^h$, the induced representation $\indpp(\chi_2^h)$ is irreducible and the lemma follows from Frobenius reciprocity.
\end{proof}

Note that in the definition $W_h = \Qpp \cdot t_2^h$, the action of $\galp$ on $\Qpp$ is semilinear, while in lemma \ref{whisind} above we extend scalars to get $\Qpp \otimes_{\Qp} W_h$ but there the action of $\galp$ on $\Qpp$ is linear. The following result is well known, see for instance proposition 5.16 of \cite{F4} and the remark which follows.

\begin{prop}\label{indiscris}
If $h \in \ZZ$, then $W_h$ is a crystalline representation of $\galp$ and if $h \leq -1$, then $\dcris(W_h) = \Qp \cdot e \oplus \Qp \cdot f$ where
\[ \Mat(\phi) = \pmat{ 0 & 1 \\ p^{-h} & 0} \quad\text{and}\quad \Fil^i \dcris(W_h) = \begin{cases} \dcris(W_h) & \text{if $i \leq 0$,} \\ \Qp \cdot e & \text{if $1 \leq i \leq -h$,} \\
\{0\} & \text{if $1-h \leq i$}. \end{cases} \]  
\end{prop}

\begin{proof}
The dual of $\indpp(\chi_2^h)$ is naturally isomorphic to $\indpp(\chi_2^{-h})$ and hence $W_h^*=W_{-h}$ so that $\dcris(W_h) = \Hom(W_{-h},\bcris)^{\galp}$. The ``inclusion map'' $e : W_{-h} \to \bcris$ is one such element and the map $f =p^h \phi \circ e$, which is given by $f : \alpha \cdot t_2^{-h} \mapsto p^h \sigma(\alpha) \cdot \phi(t_2^{-h})$ is another one which is linearly independent. The fact that $\phi^2(t_2)=pt_2$ gives us the matrix of $\phi$ in the basis $e,f$. The fact that $t_2,\phi(t_2) \in  \Fil^0 \bdr$ implies that $\Fil^0 \dcris(W_h) = \dcris(W_h)$, and the fact that $t_2 \in \Fil^1\setminus \Fil^2 \bdr$ implies that $\Fil^{-h} \dcris(W_h) = \Qp \cdot e$ and $\Fil^{1-h} \dcris(W_h) = \{0\}$.
\end{proof}

Note that there is a similar result if $h \geq 1$, but in that case the non-trivial line of the filtration is generated by $f$. 

\section{Interlude : some $p$-adic analysis}\label{psr}
Before moving on to the construction of $(\phi,\Gamma)$-modules, we prove a few simple results concerning functions belonging to the Robba ring $\calR_E$ where $E$ is a finite extension of $\Qp$. Recall that if $r<1$, then $\calR_E^{\dagger,r}$ is defined to be the set of functions $f(X) = \sum_{k \in \ZZ} a_k X^k$ with $a_k \in E$ which converge on the annulus $r \leq |X|_p < 1$ and that $\calR_E = \cup_{r < 1} \calR_E^{\dagger,r}$. This ring is endowed with a frobenius $\phi$ given by $\phi(f)(X)=f((1+X)^p-1)$, and with an action of $\Gamma=\Gal(\Qp(\mu_{p^\infty})/\Qp)$ given by $\gamma(f)(X)=f((1+X)^{\chi(\gamma)}-1)$.

The subset $\calE_E^{\dagger}$ of $\calR_E$ consisting of those functions $f(X)$ for which the sequence $\{a_k(f)\}$ is bounded is a subfield of $\calR_E$, and we write $\OO_{\calE_E}^\dagger$ for the subring of $\calE_E^\dagger$ consisting of those functions for which $|a_k(f)|_p \leq 1$ for all $k$'s. If $E=\Qp$ then we drop the subscript $E$ from the notation. These rings were studied by Lazard in \cite{L62} and more recently by Kedlaya in \cite{K5}.

Let $Q=\phi(X)/X$ and for $n \geq 1$, let $Q_n=\phi^{n-1}(Q)$ so that $Q_n$ is the minimal polynomial of $\zeta_{p^n}-1$. If $r<1$, we define $n(r)$ to be the smallest integer such that $|\zeta_{p^n}-1|_p \geq r$. We also let $t=\log(1+X)$ so that $\phi(t)=pt$ and $\gamma(t)=\chi(\gamma)t$, and recall that we have the Weierstrass product formula $t=X \cdot \prod_{n=1}^{\infty} Q_n/p$. Recall also that by an argument analogous to that of lemma I.3.2 of \cite{LB4}, we have the following result.

\begin{lemm}\label{gamstab}
Every principal ideal of $\calR_E^{\dagger,r}$ which is stable under $\Gamma$ is generated by an element of the form $\prod_{n \geq n(r)} (Q_n / p)^{a_n}$ for some $a_n$'s in $\ZZ_{\geq 0}$.
\end{lemm}

We now prove a few results which are used in the remainder of the paper.

\begin{lemm}\label{invphi}
The map $f(X) \mapsto \phi^2(f(X))/f(X)$ from $1+X\Zp\dcroc{X}$ to itself is bijective.
\end{lemm}

\begin{proof}
The map is injective because if $f(X)=1+a_k X^k + \OO(X^{k+1})$ with $a_k \neq 0$, then $\phi^2(f(X))= 1+p^{2k} a_k X^k + \OO(X^{k+1})$ so that $\phi^2(f(X))=f(X)$ if and only if $f(X)=1$. Let us now prove surjectivity. If $f(X) \in 1+X \Zp\dcroc{X}$, then $\phi^n(f(X)) \to 1$ as $n \to \infty$ and the product $\prod_{n=0}^{\infty} \phi^{2n}(f(X))$ converges to $g(X) \in 1+X \Zp\dcroc{X}$ such that $\phi^2(g(X))/g(X)=f(X)$.
\end{proof}

\begin{coro}\label{phiforgam}
If $\gamma \in \Gamma$, then there exist two uniquely determined power series $f_\gamma(X) \in 1+X\Zp\dcroc{X}$ and $g_\gamma(X) \in 1+pX\Zp\dcroc{X}$ such that 
\[ \frac{\phi^2(f_\gamma(X))}{f_\gamma(X)}=\frac{\gamma(Q_2)}{Q_2} \quad\text{and}\quad  \frac{\phi^2(g_\gamma(X))}{g_\gamma(X)}=\frac{\gamma(Q_2/Q^p)}{Q_2/Q^p}. \]
\end{coro}

\begin{proof}
We have $\gamma(Q_2)/Q_2 \in 1+X\Zp\dcroc{X}$ and $Q_2/Q^p \in 1 + p \OO_{\calE}^\dagger$ so that 
\[ \frac{\gamma(Q_2/Q^p)}{Q_2/Q^p} \in  1+X\Zp\dcroc{X} \cap 1 + p \OO_{\calE}^\dagger = 1+pX\Zp\dcroc{X}. \] 
The corollary then follows from lemma \ref{invphi}, and the fact that $g_\gamma(X) \in 1+pX\Zp\dcroc{X}$ follows from the explicit construction for the inverse of $\phi^2(\cdot)/(\cdot)$.
\end{proof}

Note that since both $f_\gamma(X)$ and $g_\gamma(X)$ are uniquely determined, we have $f_{\gamma \eta}(X) = f_\gamma(X) \gamma(f_\eta(X))$ and $g_{\gamma \eta}(X) = g_\gamma(X) \gamma(g_\eta(X))$ if $\gamma, \eta \in \Gamma$.

\begin{lemm}\label{phip}
We have $(t^{-1} \calR_E)^{\phi^2=p^{-2}} = E \cdot t^{-1}$.
\end{lemm}

\begin{proof}
If $t^{-1} f(X) \in (t^{-1} \calR_E)^{\phi^2=p^{-2}}$ then $f(X) \in \calR_E^{\phi^2=1} = E$.
\end{proof}

Let $\partial : \calR_E \to \calR_E$ be the operator defined by $\partial f(X)=(1+X)df(X)/dX$. If $f \in \operatorname{Frac} \calR_E^{\dagger,r}$ and if $n \geq n(r)$ and if $f$ has at most a simple pole at $\zeta_{p^n}-1$, we define $\res_n(f)$ to be the value at $\zeta_{p^n}-1$ of $f \cdot Q_n / \partial Q_n$ so that $\res_n(f)$ is the residue of $f$ at $\zeta_{p^n}-1$ multiplied by a suitable constant such that $\res_n(\partial Q_n/Q_n)=1$.

\begin{lemm}\label{resinz}
If $f \in \calR_E^{\dagger,r}$ and $n \geq n(r)$, then $\res_n(\partial f/f) \in \ZZ$.
\end{lemm}

\begin{proof}
If $g(X) \in \calR_E^{\dagger,r}$ is nonzero at $\zeta_{p^n}-1$ then $\partial g/g$ has no pole at $\zeta_{p^n}-1$ and hence $\res_n(\partial g/g)=0$. If $f(X) \in \calR_E^{\dagger,r}$ then we can write $f(X)=Q_n(X)^a g(X)$ where $g(X) \in \calR_E^{\dagger,r}$ is nonzero at $\zeta_{p^n}-1$ and $a$ is the order of vanishing of $f(X)$ at $\zeta_{p^n}-1$ so that $\res_n(f)=a \res_n(Q_n) + \res_n(g) = a \in\ZZ$.
\end{proof}

\section{A family of dihedral $(\varphi,\Gamma)$-modules}
We now construct the $(\phi,\Gamma)$-modules over $\calR$ associated to the representations $V_r(s)$. Since both $Q_2/Q^p$ and $g_\gamma(X)$ belong to $1 + p \OO_{\calE}^\dagger$, we have $(Q_2/Q^p)^u \in 1+ p \OO_{\calE}^\dagger$ and $g_\gamma(X)^u \in 1+ p \OO_{\calE}^\dagger$ if $u \in \Zp$. 

\begin{defi}\label{defpgm}
Given $j \in \ZZ$, we define a $(\phi,\Gamma)$-module $\dfont^0_j(u)$ over $\OO_{\calE}^\dagger$ by $\dfont^0_j = \OO_{\calE}^\dagger \cdot e \oplus \OO_{\calE}^\dagger \cdot f$ where
\[ \Mat(\phi) = \pmat{0 & 1 \\ Q_2^j (Q_2/Q^p)^u & 0} \] 
and
\[ \Mat(\gamma) = \pmat{\chi(\gamma)^j \diam{\chi(\gamma)}^u \phi(f_\gamma(X)^j g_\gamma(X)^u) & 0 \\ 0 & \chi(\gamma)^j \diam{\chi(\gamma)}^u f_\gamma(X)^j g_\gamma(X)^u}. \]
We then extend scalars to $\calR$ so as to get an \'etale $(\phi,\Gamma)$-module $\dfont_j(u) = \calR \otimes_{\OO_{\calE}^\dagger} \dfont_j^0(u)$.
\end{defi}

\begin{lemm}\label{pgmcong}
If $u_1=u_2\bmod{p^k}$, then $\dfont^0_j(u_1)=\dfont^0_j(u_2)\bmod{p^{k+1}}$.
\end{lemm}

\begin{proof}
This follows from the definition above and the fact that both $Q_2/Q^p$ and $g_\gamma(X)$ belong to $1 + p \OO_{\calE}^\dagger$ so that if $u_1=u_2\bmod{p^k}$, then $(Q_2/Q^p)^{u_1}=(Q_2/Q^p)^{u_2}\bmod{p^{k+1}}$ and $g_\gamma(X)^{u_1}=g_\gamma(X)^{u_2}\bmod{p^{k+1}}$.
\end{proof}

\begin{theo}\label{pgmisind}
We have $\dfont_j(u) = \dfont(V_j(j+(p+1)u))$ if $u \in \Zp$.
\end{theo}

\begin{proof}
By lemmas \ref{repcong} and \ref{pgmcong}, it is enough to check the isomorphism for $u$ belonging to a $p$-adically dense subset of $\Zp$, and we do so for those $u \in (p-1)\ZZ$ such that $h=j+(p+1)u \leq -1$, so that $V_j(j+(p+1)u) = V_h(h) = W_h$. Using the fact that $\phi^2(X)=Q_2QX$, we find that in the basis $(e',f') = (\phi(X)^{-u-j} e, X^{-u-j} f)$, we have 
\[ \Mat(\phi) = \pmat{0 & 1 \\ Q^{-h} & 0}, \]
and using the fact that $\gamma(X)/X=\chi(\gamma) (1+\OL(X))$ and $\gamma(Q_n)/Q_n=1+\OL(X)$, we find
\[ \Mat(\gamma) = \pmat{1+\OL(X) & 0 \\ 0  & 1+\OL(X)}. \]
In particular, $\Zp\dcroc{X} \cdot e' \oplus \Zp\dcroc{X} \cdot f'$ is a Wach module as defined in \cite[\S III.4]{LB4}. By proposition III.4.2 and theorem III.4.4 of \cite{LB4}, the $(\phi,\Gamma)$-module $\dfont_j(u)$ associated to this Wach module corresponds to a crystalline representation $V$ such that in the basis $(\overline{e}',\overline{f}')$ of  $\dcris(V)$ we have 
\[ \Mat(\phi) = \pmat{ 0 & 1 \\ p^{-h} & 0} \quad\text{and}\quad \Fil^i \dcris(V) = \begin{cases} \dcris(V) & \text{if $i \leq 0$,} \\ \Qp \cdot \overline{e}' & \text{if $1 \leq i \leq -h$,} \\
\{0\} & \text{if $1-h \leq i$}. \end{cases} \]  
By proposition \ref{indiscris} and the fact that $V \mapsto \dcris(V)$ is fully faithful by theorem 5.3.5 of \cite{FST}, we have $V = W_h$ and so $\dfont_j(u) = \dfont(V_j(j+(p+1)u))$.
\end{proof}

\section{Determination of the trianguline points}
Colmez has defined (see definitions 4.1 and 3.4 of \cite{C08}) a $p$-adic representation to be trianguline if its associated $(\phi,\Gamma)$-module over $\calR$ is an iterated extension of $(\phi,\Gamma)$-modules of rank $1$, after possibly extending scalars. In this last chapter, we determine which of the representations $V_r(s)$ are trianguline. The key point is the result below.

\begin{prop}\label{nectriang}
If $\dfont_j(u)$ is trianguline, then $(p+1)u \in \ZZ$.
\end{prop}

\begin{proof}
By definition 3.4 and proposition 3.1 of \cite{C08}, $\dfont_j(u)$ is trianguline if and only if there exist some finite extension $E/\Qp$, a continuous character $\delta : \Qp^\times \to E^\times$ and $\alpha,\beta \in \calR_E$ such that 
\[ \gamma(\alpha e + \beta f) = \delta(\gamma) (\alpha e + \beta f) \ \text{if}\  \gamma \in \Gamma \quad\text{and}\quad  \phi(\alpha e + \beta f) = \delta(p) (\alpha e + \beta f). \]
Given the formulas of definition \ref{defpgm}, the first condition implies that the ideals of $\calR_E$ generated by $\alpha$ and $\beta$ are each stable under the action of $\Gamma$, and the second condition implies that $\beta$ satisfies the equation 
\[ \phi^2(\beta) Q_2^j \left(\frac{Q_2}{Q^p}\right)^u = \delta(p)^2 \beta. \]
Let $\mu(X) \in \calR_E$ be the power series $\mu(X) = \prod_{n=1}^{\infty} Q_{2n}/p$ so that we have
\[ \phi^2(\mu) Q_2   = p \mu \quad\text{and}\quad 
\phi^2(\mu^{p+1}X^p) \frac{Q_2}{Q^p} = p^{p+1} \mu^{p+1}X^p. \]
If we apply the map $\partial(\cdot)/(\cdot)$ to the above three equations, bearing in mind that $\partial \circ \phi^2 = p^2 \phi^2 \circ \partial$, we get 
\begin{align*}
(1-p^2\phi^2) \frac{\partial \beta}{\beta} & = j \frac{\partial Q_2}{Q_2} + u \frac{\partial (Q_2/Q^p)}{Q_2/Q^p}, \\
(1-p^2\phi^2) \frac{\partial \mu}{\mu} & = \frac{\partial Q_2}{Q_2}, \\
(1-p^2\phi^2) \frac{\partial(\mu^{p+1}X^p) }{\mu^{p+1}X^p} & = \frac{\partial (Q_2/Q^p)}{Q_2/Q^p}, 
\end{align*}
so that
\[ (1-p^2\phi^2) \left( \frac{\partial \beta}{\beta} - j  \frac{\partial \mu}{\mu} - u \frac{\partial(\mu^{p+1}X^p) }{\mu^{p+1}X^p} \right) = 0. \]
Since the ideal of $\calR_E$ generated by $\beta$ is stable under the action of $\Gamma$, lemma \ref{gamstab} implies that the only possible zeroes of $\beta$ are the $\zeta_{p^n}-1$ and the same is true for $\mu$, so that $\partial \beta / \beta$ and $\partial \mu / \mu$ and $\partial(\mu^{p+1}X^p)  / \mu^{p+1}X^p$ all belong to $t^{-1} \calR_E$ since logarithmic derivatives have only simple poles. The above equation and lemma \ref{phip} imply that there exists $c \in E$ such that
\[ \frac{\partial \beta}{\beta} - j  \frac{\partial \mu}{\mu} - u \frac{\partial(\mu^{p+1}X^p) }{\mu^{p+1}X^p} = \frac{c}{t}. \]
We now fix a radius $r$ such that $\beta \in \calR_E^{\dagger,r}$ and apply the residue maps $\res_n$ of \S \ref{psr} for $n \geq n(r)$. If $n$ is odd, then $\mu(X)$ has no zero nor pole at $\zeta_{p^n}-1$ and hence $\res_n(\partial \mu / \mu)=0$, which implies that $\res_n(c/t) = \res_n(\partial \beta / \beta) \in \ZZ$ by lemma \ref{resinz} and therefore $c \in \ZZ$ since $\res_n(1/t)=\res_n(\partial t/t) =1$. If $n$ is even, then by applying once more $\res_n$ to the above equation we get $\res_n(\partial \beta/\beta) -j-(p+1)u = c$ and since both $c$ and $\res_n(\partial \beta / \beta)$ belong to $\ZZ$, we also have $(p+1)u \in \ZZ$.
\end{proof}

\begin{theo}\label{indtrig}
The representation $V_r(s)$ is trianguline if and only if $s \in \ZZ$ and $r=s \bmod{p+1}$.
\end{theo}

\begin{proof}
If $s \in \ZZ$ and $r-s=(p+1)k$, then \[ \Qpp \otimes_{\Qp} V_r(s) = \indpp(\omega_2^{r-s} \chi_2^s) =\indpp(\chi_2^s)  \otimes \omega_1^k \] where $\omega_1=\omega_2^{p+1}=\omega(\chi(\cdot))$. By proposition \ref{indiscris}, $\indpp(\chi_2^s)$ is crystalline and hence trianguline by theorem 0.8 of \cite{C08} so that $V_r(s)$ is trianguline if $s \in \ZZ$ and $r=s \bmod{p+1}$.

Assume now that $V_r(s)$ is trianguline. By combining theorem \ref{pgmisind} and proposition \ref{nectriang}, we see that $s \in \ZZ$ so that we have $V_r(s) = \indpp(\omega_2^{r-s} \chi_2^s)$ and $V_r(s)$ is potentially crystalline. By theorem 0.8 of \cite{C08}, a trianguline $p$-adic representation is potentially crystalline if and only if its restriction to $\Gal(\Qpbar/\Qp(\zeta_{p^n}))$ is crystalline for some $n \gg 0$. By restricting $V_r(s)$ to $\galpp$, we see that the fixed field of $\omega_2^{r-s}$ must lie in $\Qpp(\zeta_{p^n})$ for some $n \gg 0$ and hence that $r-s$ is divisible by $p+1$.
\end{proof}

\bibliographystyle{smfalpha}
\bibliography{dihedral}

\providecommand{\bysame}{\leavevmode ---\ }
\providecommand{\og}{``}
\providecommand{\fg}{''}
\providecommand{\smfandname}{et}
\providecommand{\smfedsname}{\'eds.}
\providecommand{\smfedname}{\'ed.}
\providecommand{\smfmastersthesisname}{M\'emoire}
\providecommand{\smfphdthesisname}{Th\`ese}
\begin{thebibliography}{Fon94b}

\bibitem[Ber02]{LB2}
{\scshape L.~Berger} -- {\og Repr\'esentations {$p$}-adiques et \'equations
  diff\'erentielles\fg}, \emph{Invent. Math.} \textbf{148} (2002), no.~2,
  p.~219--284.

\bibitem[Ber04]{LB4}
\bysame , {\og Limites de repr\'esentations cristallines\fg}, \emph{Compos.
  Math.} \textbf{140} (2004), no.~6, p.~1473--1498.

\bibitem[Ber08]{LB8}
\bysame , {\og Equations diff{\'e}rentielles $p$-adiques et
  $(\varphi,{N})$-modules filtr{\'e}s\fg}, \emph{Ast\'erisque} (2008), no.~319,
  p.~13--38.

\bibitem[BLZ04]{BLZ}
{\scshape L.~Berger, H.~Li {\normalfont \smfandname} H.~J. Zhu} -- {\og
  Construction of some families of 2-dimensional crystalline
  representations\fg}, \emph{Math. Ann.} \textbf{329} (2004), no.~2,
  p.~365--377.

\bibitem[CC98]{CC}
{\scshape F.~Cherbonnier {\normalfont \smfandname} P.~Colmez} -- {\og
  Repr\'esentations {$p$}-adiques surconvergentes\fg}, \emph{Invent. Math.}
  \textbf{133} (1998), no.~3, p.~581--611.

\bibitem[Col02]{C02}
{\scshape P.~Colmez} -- {\og Espaces de {B}anach de dimension finie\fg},
  \emph{J. Inst. Math. Jussieu} \textbf{1} (2002), no.~3, p.~331--439.

\bibitem[Col08a]{CEV}
\bysame , {\og Espaces {V}ectoriels de dimension finie et repr{\'e}sentations
  de de {R}ham\fg}, \emph{Ast\'erisque} (2008), no.~319, p.~117--186.

\bibitem[Col08b]{C08}
\bysame , {\og Repr{\'e}sentations triangulines de dimension $2$\fg},
  \emph{Ast\'erisque} (2008), no.~319, p.~213--258.

\bibitem[FM95]{FM}
{\scshape J.-M. Fontaine {\normalfont \smfandname} B.~Mazur} -- {\og Geometric
  {G}alois representations\fg}, Elliptic curves, modular forms, \& {F}ermat's
  last theorem ({H}ong {K}ong, 1993), Ser. Number Theory, I, Int. Press,
  Cambridge, MA, 1995, p.~41--78.

\bibitem[Fon90]{F90}
{\scshape J.-M. Fontaine} -- {\og Repr\'esentations {$p$}-adiques des corps
  locaux. {I}\fg}, The {G}rothendieck {F}estschrift, {V}ol.\ {II}, Progr.
  Math., vol.~87, Birkh\"auser Boston, Boston, MA, 1990, p.~249--309.

\bibitem[Fon94a]{F3}
\bysame , {\og Le corps des p\'eriodes {$p$}-adiques\fg}, \emph{Ast\'erisque}
  (1994), no.~223, p.~59--111, With an appendix by Pierre Colmez, P{\'e}riodes
  $p$-adiques (Bures-sur-Yvette, 1988).

\bibitem[Fon94b]{FST}
\bysame , {\og Repr\'esentations {$p$}-adiques semi-stables\fg},
  \emph{Ast\'erisque} (1994), no.~223, p.~113--184, P{\'e}riodes $p$-adiques
  (Bures-sur-Yvette, 1988).

\bibitem[Fon04]{F4}
\bysame , {\og Arithm\'etique des repr\'esentations galoisiennes
  {$p$}-adiques\fg}, \emph{Ast\'erisque} (2004), no.~295, p.~xi, 1--115,
  Cohomologies $p$-adiques et applications arithm{\'e}tiques. III.

\bibitem[Ked05]{K5}
{\scshape K.~S. Kedlaya} -- {\og Slope filtrations revisited\fg}, \emph{Doc.
  Math.} \textbf{10} (2005), p.~447--525 (electronic).

\bibitem[Kis03]{K3}
{\scshape M.~Kisin} -- {\og Overconvergent modular forms and the
  {F}ontaine-{M}azur conjecture\fg}, \emph{Invent. Math.} \textbf{153} (2003),
  no.~2, p.~373--454.

\bibitem[Laz62]{L62}
{\scshape M.~Lazard} -- {\og Les z\'eros des fonctions analytiques d'une
  variable sur un corps valu\'e complet\fg}, \emph{Inst. Hautes \'Etudes Sci.
  Publ. Math.} (1962), no.~14, p.~47--75.

\bibitem[Maz89]{M89}
{\scshape B.~Mazur} -- {\og Deforming {G}alois representations\fg}, Galois
  groups over {${\bf Q}$} ({B}erkeley, {CA}, 1987), Math. Sci. Res. Inst.
  Publ., vol.~16, Springer, New York, 1989, p.~385--437.

\end{thebibliography}
\end{document}